\documentclass[12pt,a4paper]{amsart}
\usepackage{mathtools,amssymb,amsmath,amsopn,amsthm,graphicx,MnSymbol,centernot,stmaryrd}
\usepackage{tikz}
\usetikzlibrary{matrix}
\begin{document}

\newtheorem{definition}{Definition}[subsection]
\newtheorem{definitions}[definition]{Definitions}
\newtheorem{deflem}[definition]{Definition and Lemma}
\newtheorem{lemma}[definition]{Lemma}
\newtheorem{pro}[definition]{Proposition}
\newtheorem{theorem}[definition]{Theorem}
\newtheorem{cor}[definition]{Corollary}
\newtheorem{cors}[definition]{Corollaries}
\theoremstyle{remark}
\newtheorem{remark}[definition]{Remark}
\theoremstyle{remark}
\newtheorem{remarks}[definition]{Remarks}
\theoremstyle{remark}
\newtheorem{notation}[definition]{Notation}
\theoremstyle{remark}
\newtheorem{example}[definition]{Example}
\theoremstyle{remark}
\newtheorem{examples}[definition]{Examples}
\theoremstyle{remark}
\newtheorem{dgram}[definition]{Diagram}
\theoremstyle{remark}
\newtheorem{fact}[definition]{Fact}
\theoremstyle{remark}
\newtheorem{illust}[definition]{Illustration}
\theoremstyle{remark}
\newtheorem{rmk}[definition]{Remark}
\theoremstyle{definition}
\newtheorem{que}[definition]{Question}
\theoremstyle{definition}
\newtheorem{ques}[definition]{Questions}
\theoremstyle{definition}
\newtheorem{conj}[definition]{Conjecture}
\newtheorem{por}[definition]{Porism}

\newcommand\A{\mathcal{A}}
\newcommand\B{\mathcal{B}}
\newcommand\C{\mathcal{C}}
\newcommand\D{\mathcal{D}}
\newcommand\E{\mathcal{E}}
\newcommand\F{\mathcal{F}}
\newcommand\G{\mathcal{G}}
\newcommand\J{\mathcal{J}}
\newcommand\K{\mathcal{K}}
\newcommand\LL{\mathcal{L}}
\newcommand\M{\mathcal{M}}
\newcommand\T{\mathcal{T}}
\newcommand\U{\mathcal{U}}
\newcommand\V{\mathcal{V}}

\newcommand\EX{\mathbb{EX}}
\newcommand\REG{\mathbb{REG}}
\newcommand\DEF{\mathbb{DEF}}

\newcommand\ABEX{\mathbb{ABEX}}
\newcommand\COH{\mathbb{COH}}
\newcommand\DEFA{\mathbb{DEF}^+}

\newcommand\Ab{\mathrm{Ab}}
\newcommand\Ex{\mathrm{Ex}}
\newcommand\Abs{\mathrm{Abs}}
\newcommand\Mod[1]{\operatorname{Mod}\mbox{-}#1}
\newcommand\fpmod[1]{\mathrm{mod}\mbox{-}#1}
\newcommand\Modl[1]{#1\mbox{-}\operatorname{Mod}}
\newcommand\fpmodl[1]{#1\mbox{-}\mathrm{mod}}

\newcommand\Set{\operatorname{Set}}
\newcommand\Cat{\operatorname{Cat}}
\newcommand\Reg{\operatorname{Reg}}
\newcommand\Fun{\operatorname{Fun}}
\newcommand\fun{\operatorname{fun}}
\newcommand\Pts{\operatorname{Pts}}
\newcommand\reg[1]{#1^{\operatorname{reg}}}
\newcommand\fp[1]{#1_{\operatorname{fp}}}
\newcommand\PSh[1]{\operatorname{PSh}(#1)}
\newcommand\Sh[2]{\operatorname{Sh}(#1,#2)}
\newcommand\Str{\operatorname{Str}}
\newcommand\Lex{\operatorname{Lex}}
\newcommand{\colim}{\operatorname{colim}}

\newcommand\Ind{\operatorname{Ind}}
\newcommand\Id{\operatorname{Id}}

\renewenvironment{proof}{\noindent {\bf{Proof.}}}{\hspace*{3mm}{$\Box$}{\vspace{9pt}}}
\title{Definable Categories}
\author{Amit Kuber}
\address{
A. Kuber\newline
Department of Mathematics and Statistics\newline
Masaryk University, Faculty of Sciences\newline
Kotl\'{a}\v{r}sk\'{a} 2, 611 37 Brno, Czech Republic
}
\email{expinfinity1@gmail.com}
\author{Ji\v{r}\'{i} Rosick\'{y}}
\address{
J. Rosick\'{y}\newline
Department of Mathematics and Statistics\newline
Masaryk University, Faculty of Sciences\newline
Kotl\'{a}\v{r}sk\'{a} 2, 611 37 Brno, Czech Republic
}
\email{rosicky@math.muni.cz}
\thanks{Both authors were supported by the Grant Agency of the Czech Republic under the grant P201/12/G028.}

\begin{abstract}
We introduce the notion of a definable category--a category equivalent to a full subcategory of a locally finitely presentable category that is closed under products, directed colimits and pure subobjects. Definable subcategories are precisely the finite-injectivity classes. We prove a $2$-duality between the $2$-category of small exact categories and the $2$-category of definable categories, and provide a new proof of its additive version. We further introduce a third vertex of the $2$-category of regular toposes and show that the diagram of $2$-(anti-)equivalences between three $2$-categories commutes; the corresponding additive triangle is well-known.
\end{abstract} 
\maketitle

\section{Introduction}
This paper belongs to the realm of categorical logic where one studies the interplay between syntax and semantics using the language of category theory, that is usually presented in the form of a dual adjunction as below: 
\begin{equation}\label{SynSemAdj}
(-)\mbox{-}\mathrm{Mod}(\Set):\mathrm{Theories}^{op}\leftrightarrows\mathrm{Exactness\ properties}:\mathrm{Exact}(-,\Set)
\end{equation}
In Makkai's terminology, for a given fragment of first-order logic,  such an adjunction is a consequence of a `logical doctrine', namely the fact that certain limits, colimits or combinations thereof commute with/distribute over others in the base category $\Set$ of sets. The above notion of `exactness property' on the semantics side precisely captures these commutativity/distributivity conditions, and the exact functors preserve all combinations of limits and colimits present in the exact categories.

For lex (i.e., finitely complete) categories, the Gabriel-Ulmer duality \cite{GU} is a full $2$-duality between the $2$-category $\mathrm{LEX}$ of small lex (i.e., finitely complete) categories, lex functors and natural transformations on the one hand, and the $2$-category $\mathrm{LFP}$ of locally finitely presentable categories, finitary right adjoint functors and natural transformations on the other:
\begin{equation*}
\Lex(-,\Set):\mathrm{LEX}^{op}\leftrightarrows\mathrm{LFP}:\fp{(-)}
\end{equation*}
Using \cite[Corollary~4.7]{ALR} it can be shown that, for a locally finitely presentable category $\K$, a functor $F:\K\to\Set$ preserves all limits and directed colimits (in notation, $F\in(\K,\Set)^{\lim\to}$) iff there is $K\in\fp\K$ such that $F\simeq\K(K,-)$ (also see \cite[p.101]{MakStone}). Thus the above duality can be rewritten in the following form:
\begin{equation}\label{CartDual}
\Lex(-,\Set):\mathrm{LEX}^{op}\leftrightarrows\mathrm{LFP}:(-,\Set)^{\lim\to}
\end{equation}

Cartesian logic is the internal logic of lex categories (see \cite[Definition~D1.3.4]{Ele} for the full syntactic definition). For this fragment of first-order logic the $2$-adjunction of \eqref{SynSemAdj} restricts to the full $2$-duality of \eqref{CartDual}. In order to achieve this, one needs to characterize the categories in the image of the $2$-functors in the adjunction given by \eqref{SynSemAdj}. The `logical doctrine' for cartesian logic states that finite limits commute with all limits as well as with directed colimits in $\Set$. The `exactness property' for this logic is captured by precontinuous categories of \cite{ALR}, namely the categories containing all limits and directed colimits in which directed colimits commute with finite limits and products distribute over directed colimits. The locally finitely presentable categories which are models of cartesian theories are precisely the precontinuous categories satisfying a smallness conditions (see \cite[Theorem~5.8]{ALR}). On the syntactic side, a cartesian theory can be recovered from the category of its models up to Morita equivalence; such Morita equivalence classes are in one-to-one correspondence with small lex categories.

\subsection{The regular case: syntax-semantics duality}
One of the main contributions of this paper is Theorem \ref{EXDEFDual}, the following analogue of \eqref{CartDual} for regular logic--a fragment of first-order logic where formulas are constructed only using truth, finitary conjunctions and existential quantifiers (Definition \ref{regtheory}).
\begin{equation}\label{RegDual}
\Reg(-,\Set):\EX^{op}\leftrightarrows\DEF:(-,\Set)^{\prod\to}
\end{equation}
The most important new concept in this paper is that of a \emph{definable category} which we borrowed from the additive world. This notion is `relative' in nature, i.e., such categories are (equivalent to) subcategories of locally finitely presentable categories closed under products, directed colimits and pure subobjects. It is an open question to find an `absolute' characterization of definable categories, where one provides a complete list of verifiable category-theoretic properties. The $2$-category $\DEF$ has definable categories as objects, interpretation functors (i.e., functors preserving products and directed colimits) as $1$-morphisms and natural transformations as $2$-morphisms. The doctrine for regular logic is the statement that finite limits and coequalizers commute with products and directed colimits in $\Set$. The exactness properties for regular logic are captured by the concept of predefinable categories introduced in Section \ref{Predef}; they are categories with products and directed colimits, and where products distribute over directed colimits.

On the left side of the duality is the $2$-category $\EX$ of small (Barr-)exact categories, regular functors and natural transformations. Morita equivalence classes of regular theories are in one-to-one correspondence with exact categories. The concept of an exact category was introduced by Barr in \cite{Barr} with the intention to codify what is common in $\Set$ and the category $\Ab$ of abelian groups. The existence of the exact completion of a lex category was shown by Carboni and Celia Magno in \cite{CM}. One half of the duality \eqref{RegDual} was proved by Makkai as (the finitary version of) \cite[Theorem~5.1]{MakInf} under the name of `strong conceptual completeness theorem' for regular logic, where one recovers (the Morita equivalence class of) a theory from its category of models. For any regular cardinal $\kappa$, Hu \cite[Theorem~5.10]{HuDual} generalized Makkai's result and proved the duality between $\kappa$-accessible categories with products and $\kappa$-Barr-exact categories with enough projectives (there such categories are called $\kappa$-Barr-exact accessible).

Definable categories are accessible with variable rank of accessibility, contain directed colimits, but their absolute characterization is unknown. On the other hand, the categories on the semantics side of Hu's duality, namely the $\kappa$-accessible categories with products, have an absolute definition but such categories may not have directed colimits. Thus, in some sense, our duality covers a cross-section of Hu's duality as $\kappa$ varies.

\subsection{The triangle of $2$-(anti-)equivalences}
We add a third $2$-category of regular toposes--the classifying toposes of regular theories--in the picture to obtain (anti-)equivalences between three $2$-categories. The fact that a theory in a well-studied fragment of first-order logic can be recovered, up to Morita equivalence, from its classifying topos is well-known (see \cite[\S D3]{Ele}). This gives a full $2$-duality between the $2$-category of (equivalence classes of) theories and the $2$-category of Grothendieck toposes of appropriate kind. Thus a theory manifests in three different yet equivalent forms. The theory itself represents `syntax', the category of models represents `semantics', whereas the classifying topos represents a topos containing a model of a given theory which is generic amongst its models in toposes.

The notion of injectivity gives rise to regular theories. In fact it is possible to characterize definable categories as (finite-)injectivity classes. Such a characterization allows one to extract the classifying topos of the underlying theory as the category of certain functors on the definable category. On the other hand, the points of the classifying topos (i.e., geometric morphisms from the base topos $\Set$) are precisely the $\Set$-models of the theory. This $2$-equivalence between the $2$-category of categories of models and the $2$-category of classifying toposes can be derived from \cite[Lemma~15.1]{Car} proved by Caramello. This discussion can be summarized as the commutativity of the triangle in Figure \eqref{Scheme}.

\begin{figure}[h]
\centering
\begin{tikzpicture}
\matrix (m) [matrix of math nodes,row sep=4em,column sep=4em,minimum width=6em]
  {\EX & \ & \DEF \\ \ & \REG & \ \\};
\path[-stealth]
    (m-1-1) edge[transform canvas={yshift=0.4ex}] node [above] {$\simeq^{op}$} (m-1-3)
    (m-1-3) edge[transform canvas={yshift=-0.4ex}] node [below] {} (m-1-1)
    	(m-1-3)	edge[transform canvas={yshift=0.4ex},pos=0.75] node {} (m-2-2)
    	(m-2-2) edge[transform canvas={yshift=-0.4ex}] node [sloped,below] {$\simeq$} (m-1-3)
    	(m-1-1)	edge[transform canvas={yshift=0.4ex},pos=0.75] node {} (m-2-2)
    	(m-2-2) edge[transform canvas={yshift=-0.4ex}] node [sloped,below] {$\simeq^{op}$} (m-1-1);
\end{tikzpicture}
\caption{The non-additive regular case: Scheme}\label{Scheme}
\end{figure}
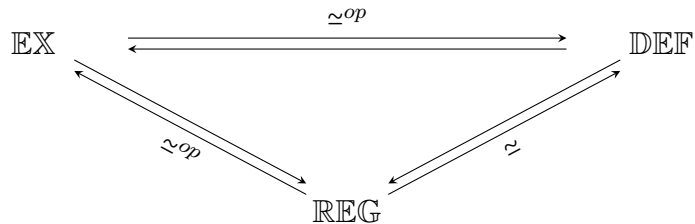

Coherent logic is a fragment of first-order logic richer than the regular one in the sense that it also uses falsity and finitary disjunctions for the construction of formulas. The strong conceptual completeness for coherent logic was proved by Makkai in \cite[Theorem~4.1]{MakStone}. In categorical terms, one replaces small exact categories by small pretoposes which also contain finite coproducts. The doctrine for coherent logic is the statement that finite limits, finite coproducts and coequalizers commute with directed colimits and ultraproducts in $\Set$ while the exactness properties for this logic are captured by `ultracategories' of Makkai.

\subsection{The additive version of the triangle}
The additive coherent analogue of the diagram in Figure \eqref{Scheme} is known due to Prest and others (see \cite[Theorem~1.1]{Pre32}) and served as one of the motivations for this paper. A small exact category is replaced by a small abelian category, a definable additive category is what it should be, and the additive analogue of a regular topos is a locally coherent additive category. Prest anticipated \cite[\S 4]{Pre32} what the non-additive coherent analogue of their diagram should look like where one considers small pretoposes instead of small exact categories. The correspondence between additive coherent case and non-additive regular case might look surprising at a first glance but it is natural because additive pretopos = additive exact category = abelian category due to the presence of biproducts in additive categories. For similar reasons, the notion of injectivity with respect to a finite cone will coincide with the notion of injectivity with respect to a single morphism. This means that on the additive side, the regular and the coherent cases coincide whereas they are distinct on the non-additive side.

A distinctive feature of the additive case is the presence of internal duality at each vertex in the diagram that stems from the fact that the dual of a small abelian category is again abelian. This fact was heavily exploited in their proofs of $2$-dualities. Motivated by the lack of internal dualities in non-additive case, we reprove as Theorem \ref{ABEXDEFADual} the duality between small abelian categories and definable additive categories using methods developed for the non-additive case. One of the basic tools for this is Freyd's \cite{Freyd} free abelian completion of a small preadditive category which precedes \cite{CM} and motivated the exact completion of a small lex category.

We freely use the terminology and notations from \cite{AR} throughout this manuscript.

\section*{Acknowledgments}
The first author would like to express sincere thanks to Olivia Caramello and Mike Prest. Caramello brought this project to his attention, pointed to the work of Makkai and Hu, and we had many discussions at the preliminary stage of the project. Prest explained the additive case to him and had many discussions on the topic.

\section{Definable categories in locally finitely presentable categories}
\subsection{Definition and characterizations}
Recall from \cite{AR} that, for a regular cardinal $\lambda$, a locally $\lambda$-presentable category is a cocomplete category with a set $\mathcal{A}$ of $\lambda$-presentable objects such that every object is a $\lambda$-directed colimit of objects of $\mathcal{A}$.

Below we give the most important definition; the terminology is inspired by and related to the model-theoretic notion of definability.
\begin{definitions}
A full subcategory $\LL$ of a locally finitely presentable category $\K$ is a \emph{definable subcategory} if it is closed in $\K$ under products, directed colimits and pure subobjects. Furthermore, a category is \emph{definable} if it is equivalent to a definable subcategory of a locally finitely presentable category.

A functor between two definable categories is an \emph{interpretation functor} if it preserves products and directed colimits.
\end{definitions}

Let $\DEF$ denote the $2$-category of definable categories whose $1$-morphisms are interpretation functors and $2$-morphisms are natural transformations.

Definable subcategories coincide with certain small-injectivity classes in locally finitely presentable categories. A class $\mathcal{A}$ of objects in a locally presentable category $\K$ is called a (small-)injectivity class provided there is a (set) collection $\mathcal{M}$ of morphisms in $K$ such that $\mathcal{A}$ is precisely the collection of all objects in $\K$ that are $h$-injective for all $h\in\mathcal{M}$.

Recall that a category is called \emph{weakly locally ($\lambda$-)presentable} if it is ($\lambda$-)accessible and has weak colimits.
\begin{theorem}\cite[Theorem~4.11]{AR}\label{CharWeakLocPres}
The following are equivalent for a category $\LL$:
\begin{enumerate}
\item $\LL$ is equivalent to a small-injectivity class in a locally presentable category $\K$.
\item $\LL$ is weakly locally presentable.
\item $\LL$ is an accessible category with products.
\item $\LL$ is equivalent to a weakly reflective, accessibly embedded subcategory of a presheaf category.
\end{enumerate}
\end{theorem}

A sharper cardinal-specific characterization is provided by the following theorem.
\begin{theorem}\cite[Theorem~2.2]{RAB}
Let $\K$ be a locally $\lambda$-presentable category. A full subcategory of $\K$ is a $\lambda$-injectivity class (i.e., injectivity class w.r.t. a set of morphisms having $\lambda$-presentable domains and codomains) iff it is closed under products, $\lambda$-directed colimits and $\lambda$-pure subobjects.
\end{theorem}

We obtain the following when $\lambda=\aleph_0$.
\begin{cor}
A category is definable iff it is equivalent to a finite-injectivity class in a locally finitely presentable category.
\end{cor}

It is possible to replace directed colimits by reduced products in the definition of definable categories in the presence of the remaining two conditions.
\begin{pro}
A full subcategory $\LL$ of a locally finitely presentable category $\K$ is definable iff it is closed in $\K$ under products, reduced products and pure subobjects.
\end{pro}

\begin{proof}
Necessity is clear since reduced products can be written as directed colimits of certain products.

For sufficiency we only need to show that a class closed under products, reduced products and pure subobjects is also closed under directed colimits. Let $(I,\leq)$ be a directed set and let $\D:I\to\LL$ be a diagram; equivalently we have $((D_i)_{i\in I};(d_{ij})_{i\leq j})$. We know that $D=\colim\D$ exists in $\K$ since $\K$ is cocomplete. Denote the canonical maps $D_i\to D$ by $d_i$.

Consider the embedding $h_i:D_i\to\prod_{j\geq i}D_j$ given by $h_i(x):=(d_{ij}(x))_{j\geq i}$. Let $\F$ denote the smallest filter containing the principal upper sets $\{j:j\geq i\}$ for each $i\in I$. Since $I$ is directed every element of $\F$ contains a principal upper set and hence the principal upper sets form a cofinal set in $\F^{op}$. Let $\prod_\F D_i$ denote the reduced product $\prod_{i\in I}D_i/\F$.

There are natural maps $H_{ik}:\prod_{j\geq i}D_j\to\prod_{l\geq k}D_l$ for $i\leq k$ and $H_i:\prod_{j\geq i}D_j\to\prod_\F D_i$ which satisfy $H_{ik}h_i=h_k d_{ik}$ and $H_kH_{ik}=H_i$. By cofinality of principal upper sets in $\F^{op}$, we see that $\prod_\F D_i$ is in fact the colimit of the system $((\prod_{j\geq i}D_j)_{i\in I};(H_{ik})_{i\leq k})$. By the universal property of the colimit, there exists a unique map $h:D\to\prod_\F D_i$ such that $H_ih_i=h d_i$ for each $i\in I$. Consequently $h$ is the colimit of the maps $h_i$ in the arrow category $\K^2$.

Now each $h_i$ is split and hence a pure embedding. Since a directed colimit of pure embeddings is pure \cite[Proposition~2.30]{AR}, we obtain that $h$ is pure. Since $\LL$ contains reduced products and is closed under pure subobjects, we conclude that $D\in\LL$.
\end{proof}

Observe that all the above characterizations of definable categories are relative with respect to an embedding in a locally finitely presentable category. In general it is hard to check whether a given category is definable.
\begin{que}(cf. \cite[Ch.~17]{PreAMS})
Is there an ``intrinsic'' characterization of definable categories?
\end{que}
We will provide a representation theorem for definable categories in Theorem \ref{DefExDef}. Unfortunately the verification of the categorical equivalence provided by that theorem is not intrinsic in nature.

In the next two sections we provide a partial answer to this question by studying two classes of categories containing definable categories.

\subsection{Predefinable categories}\label{Predef}
Let $\LL$ be a category with products and directed colimits. Recall that \emph{products distribute over directed colimits} in $\LL$ when, given a set $(\D_i:D_i\to\A)_{i\in I}$ of directed diagrams in $\LL$, the canonical morphism
$\colim \prod_{i\in I}\D_i\to\prod_{i\in I}\colim\D_i$ is an isomorphism where the ``product diagram" $\prod_{i\in I}\D_i:\prod_{i\in I}D_i\to\A$ is defined by $(d_i)\longmapsto\prod\D_id_i$.

Let $\Ind\LL$ be the free completion of a category $\LL$ under directed colimits and $\eta_\LL:\LL\to\Ind\LL$ be the embedding. If $\LL$ has directed colimits, we get a functor $C_\LL:\Ind\LL\to\LL$ preserving directed colimits and satisfying $C_\LL\circ\eta_\LL\simeq\Id_\LL$.

\begin{definition}
A category is called \emph{predefinable} if it has products and directed colimits and products distribute over directed colimits.
\end{definition}

\begin{theorem}
A category $\LL$ is predefinable if and only if $C_\LL$ preserves products.
\end{theorem}
This is $(\ast)$ in the proof of \cite[Theorem~2.1]{ARV}.

Accessible categories with directed colimits are precisely reflective and accessibly embedded full subcategories of finitely accessible categories. In fact, if $\LL$ is accessible then $\LL\simeq\Ind_\lambda\C$, the free completion of a small category $\C$ under $\lambda$-directed colimits for some regular cardinal $\lambda$. Then the inclusion $E:\LL\to\Ind\C$ is the desired accessible full embedding into a finitely accessible category with a left adjoint $H:\Ind\C\to\Ind_\lambda\C$ such that $H\circ E\simeq\Id_\LL$.

Reflective full subcategory whose reflector preserves limits is called a \emph{complete localization}. We say that it is a \emph{complete prelocalization} if the reflector preserves products.

\begin{lemma}\label{le2.3}
Complete prelocalizations of predefinable categories are predefinable.
\end{lemma}
See \cite[Example~2.3(2)]{ARV} for a proof. Also see \cite[\S 5]{ALR}.

\begin{theorem}  
Accessible predefinable categories are precisely the accessible complete prelocalizations of weakly locally finitely presentable categories.
\end{theorem}
\begin{proof}
Let $\LL$ be a $\lambda$-accessible predefinable category and denote by $\C$ its small full subcategory representing all $\lambda$-presentable objects. We show that $\Ind\C$ is weakly locally finitely presentable and $E:\LL\to\Ind\C$ above is a complete prelocalization. Following \cite[Theorem~4.11]{AR}, $\LL$ is weakly locally $\lambda$-presentable and, following part I of the proof of \cite[Theorem~4.13]{AR}, $\C$ has weak $\lambda$-small colimits. Moreover, following this proof, $\Ind\C$ is sketchable by a limit-epi sketch. Thus it is weakly locally finitely presentable. Following the proof of \cite[Theorem~2.7]{ARV}, left adjoint $H$ to $E$ preserves products. Thus $E$ is a complete prelocalization.

Conversely, let $\LL$ be an accessible complete prelocalization of a weakly locally finitely presentable category $\K$. Let $\A$ be a small full subcategory of $\K$ representing all finitely presentable objects. Then $\K\simeq\Ind\A$ and consider the full embedding $G:\K\to\Set^{\A^{op}}$. Following \cite[Theorem~4.11]{AR}, $G$ makes $\K$ a weakly reflective subcategory of $\Set^{\A^{op}}$ closed under directed colimits. Hence $\K$ is closed under products in $\Set^{\A^{op}}$ (see \cite[Remark~4.5(3)]{AR}) and thus $\K$ is predefinable. The category $\LL$ is accessible as a reflective full subcategory of an accessible category and predefinable following Lemma \ref{le2.3}.
\end{proof}

\begin{rmk}\label{re2.5}
It follows from \cite[Remark~2.5,\,Theorem~2.7]{ARV} that any complete localization of a locally finitely presentable category is an accessible predefinable category.
\end{rmk}

Recall that $\Ind\C$ is weakly locally finitely presentable if and only if $\C$ has weak finite colimits. The following result provides us with a subclass of the class of definable categories where absolute characterization is known.
\begin{pro}
Every weakly locally finitely presentable category is definable.
\end{pro}
\begin{proof}
Let $\Ind\C$ be weakly locally finitely presentable and consider the inclusion $G:\Ind\C\to\Set^{\C^{op}}$. Since $\Ind\C$ and $\Set^{\C^{op}}$ are finitely accessible and $G$ preserves directed colimits and finitely presentable objects, $\Ind\C$ is a finite-injectivity class in $\Set^{\C^{op}}$ (see the proof of (ii)$\Rightarrow$(i) of \cite[Theorem~4.8]{AR}). 
\end{proof}
 
\subsection{Weakly definable categories}\label{WeakDef}
\begin{definition}
We say that a full subcategory of a locally finitely presentable category is \emph{weakly definable} if it is closed under products and directed colimits. Furthermore, a category is \emph{weakly definable} if it is equivalent to a weakly definable subcategory of a locally finitely presentable category.
\end{definition}

\begin{definition}
Suppose $\K$ is a locally finitely presentable category. An $(\omega,\lambda)$-\emph{injectivity class} in $\K$ is a full subcategory of $\K$ consisting of objects injective to a set of morphisms having finitely presentable domain and $\lambda$-presentable codomain.
\end{definition}

\begin{lemma}\label{le3.4} 
Any $(\omega,\lambda)$-injectivity class is an accessible weakly definable category.
\end{lemma}
\begin{proof}
Let $\LL$ be an $(\omega,\lambda)$-injectivity class in a locally finitely presentable category $\K$. Then $\LL$ is accessible by \cite[Theorem~4.8]{AR} and it is obviously closed under products and directed colimits. 
\end{proof}

In this way we get plenty of examples of accessible predefinable categories.

\begin{definition}
We say that a morphism $f:K\to L$ in a locally finitely presentable category $\K$ is $(\omega,\lambda)$-\emph{pure} if, given morphisms $g:A\to B,\ u:A\to K,\ v:B\to L$ such that $f u=v g$ where $A$ is finitely presentable and $B$ is $\lambda$-presentable, there exists $t:B\to K$ such that $t g=u$.
\end{definition}

Any $(\omega,\lambda)$-pure morphism is pure and thus it is a monomorphism.

\begin{rmk}
Any $(\omega,\lambda)$-injectivity class in a locally finitely presentable category is closed under products, directed colimits and $(\omega,\lambda)$-pure subobjects. We do not know whether the converse is true.
\end{rmk}

The following implications are obvious.

Definable $\Rightarrow$ Weakly definable $\Rightarrow$ Accessible predefinable

\begin{ques}
\begin{enumerate}
\item Is every weakly definable category definable?
\item Is a complete definable category locally finitely presentable?
\end{enumerate}
\end{ques}
We expect a negative answer to the first problem because Lemma \ref{le3.4} yields a lot of weakly definable categories $\LL\subseteq\K$ which are not definable in this embedding to $\K$. But it seems to be hard to prove that they are not definable.

Now we give an example of an accessible predefinable category that fails to be weakly definable.
\begin{example}\label{ex4.8}
The lattice $[0,1]$ is an accessible predefinable category (see \cite[Remark~2.6]{ARV} and Remark \ref{re2.5}).

Assume that $[0,1]$ is a full subcategory of a locally finitely presentable category $\K$ closed under products and filtered colimits. The category $\K\simeq\Lex(\fp\K^{op},\Set)$ is a finite-orthogonality class in $(\fp\K^{op},\Set)$, which in turn is a finite-orthogonality class in $\Str\Sigma$ for a finitary $S$-sorted relational signature $\Sigma$ for some set $S$ as described in \cite[Example~1.41]{AR}. Thus $\K$ is a finite orthogonality class in $\Str\Sigma$.

Consider the full embedding $F:[0,1]\to\K\to\Str\Sigma$. Since $Fa=F(a\wedge a)=Fa\times Fa$, $\mathrm{id}_{Fa\times Fa}$ is the only morphism $Fa\times Fa\to Fa\times Fa$. Since the identity on $Fa\times Fa$ equals to the symmetry of  $Fa\times Fa$, the $\Sigma$-structure $Fa$ is a subobject of a terminal $\Sigma$-structure, i.e., the underlying $S$-sorted set $(X_s)_{s\in S}$ of $Fa$ has all $X_s$ either empty or singletons. Thus all compositions $G:[0,1]\to\K\to\Str\Sigma \to\Set^S\to\Set$ are subfunctors of the constant functor on the one-element set $1$.

Consider the set $J$ of all $0\leq a\leq 1$ such that $Ga=\emptyset$. Since $G$ preserves products, the complement of $J$ in $[0,1]$ has the smallest element $j$. Since $G$ preserves directed colimits, we have $j=0$. Thus $G$ is the constant on $1$. Hence the composition $H:[0,1]\to\K\to\Str\Sigma\to\Set^S$ is the constant on the terminal object of $\Set^S$. Any relation symbol $R$ of $\Sigma$ induces a subfunctor of the constant functor $[0,1]\to\Set$ on $1$ which preserves products and filtered colimits. In the same way as above, this subfunctor is the constant functor on $1$. Thus $F$ is the constant functor on the terminal object of $\Str\Sigma$, which is a contradiction. Thus we have proved that $[0,1]$ is not weakly definable.
\end{example}

\section{Duality between $\DEF$ and $\EX$}
The goal of this section is to prove Theorem \ref{DefExDef} which states that definable categories ``are the same as'' the categories of regular functors on small exact categories. This statement gives a $2$-duality theorem.
\subsection{Exact categories and exact completions}
We define exact categories and introduce the concept of the exact completion of a small lex category that was originally introduced in \cite{CM}.
\begin{definition}
In a lex category $\C$ an (internal) equivalence relation on an object $X$ is a subobject $(p_1,p_2):R\rightarrowtail X\times X$ equipped with the following morphisms:
\begin{itemize}
\item (reflexivity) $r:X\to R$ such that $p_1\circ r=p_2\circ r=\mathrm{id}_X$.
\item (symmetry) $s:R\to R$ such that $p_1\circ s=p_2$ and $p_2\circ s=p_1$.
\item (transitivity) $t:R\times_XR\to R$ with projection maps $q_1,q_2:R\times_XR\to R$ such that $p_1\circ q_1=p_1\circ t$ and $p_2\circ q_2=p_2\circ t$.
\end{itemize}
\end{definition}

The kernel pair of a morphism (i.e., the pullback along itself) in a lex category gives an equivalence relation on its domain.

\begin{definitions}
A category $\B$ is regular if it is lex, it admits coequalizers of kernel pairs of all its morphisms, and in which regular epimorphisms are stable under pullbacks.

A functor between $b:\B_1\to\B_2$ two regular categories is regular if it is lex and preserves regular epimorphisms.

A regular category $\B$ is exact (in the sense of Barr) if every equivalence relation is a kernel pair of some morphism.
\end{definitions}
We denote by $\EX$ the $2$-category of small exact categories whose $1$-morphisms are regular functors and $2$-morphisms are natural transformations.

\begin{definition}
The free exact completion of a small lex category $\C$ is an exact category $\C_{ex/lex}$ together with a lex functor $h:\C\to\C_{ex/lex}$ such that $h$ has the following universal property: for any exact category $\B$, there is an equivalence of categories $-\circ h:\Reg(\C_{ex/lex},\B)\to\Lex(\C,\B)$.
\end{definition}

\begin{theorem}\label{ExdefIsDef}
For a small exact category $\B$, the category $\Reg(\B,\Set)$ of regular functors on $\B$ is definable. Moreover, $\B\simeq(\Reg(\B,\Set),\Set)^{\prod\to}$, the category of functors from $\Reg(\B,\Set)$ to $\Set$ that preserve products and directed colimits.
\end{theorem}

\begin{proof}
Let $\B$ denote a small exact category. The embedding of $\Reg(\B,\Set)$ into the category $\Lex(\B,\Set)$ of lex(= finite limit preserving) functors is fully faithful. Recall from the Gabriel-Ulmer duality that the latter category is locally finitely presentable (see \cite[Theorem~1.46]{AR}).

Let $B_1\xrightarrow{h}B_2$ be a regular epimorphism in $\B$. It is easy to see that a lex functor $F:\B\to\Set$ takes the regular epimorphism $h$ to a regular epimorphism(=surjection) in $\Set$ iff $F$ is injective to the natural transformation $(-)\circ h:\B(B_2,-)\to\B(B_1,-)$ in $\Lex(\B,\Set)$. Taking the set $\M$ of all such natural transformations corresponding to regular epimorphisms in $\B$, we obtain that the regular functors are precisely those which are injective with respect to each morphism in $\M$. Moreover since all representable functors are finitely presentable, we have described $\Reg(\B,\Set)$ as a finite-injectivity class in the locally finitely presentable category $\Lex(\B,\Set)$ and hence it is definable.

The second part is the finitary version of \cite[Theorem~5.1]{MakInf}.
\end{proof}

\begin{cor}\label{ExComplLex}
For a small lex category $\C$, $\C_{ex/lex}\simeq(\Lex(\C,\Set),\Set)^{\prod\to}$.
\end{cor}

\subsection{Exactly definable categories}
The aim of this section is to show that a definable category can be recovered from its definable structure, i.e., the $\Set$-valued functors which preserve products and directed colimits.

\begin{definition}(cf. Definition \ref{AddExdef})
A category $\LL$ is exactly definable if there exists a small exact category $\B$ such that $\LL$ is equivalent to the category $\Reg(\B,\Set)$ of regular functors on $\B$.
\end{definition}

Let $\A$ be any category with product and directed colimits. Since products and directed colimits commute with finite limits and regular epimorphisms in $\Set$, the category $\A^+:=(\A,\Set)^{\prod\rightarrow}$ of functors preserving products and directed colimits is an exact category. We have adopted notation $\A^+$ from \cite{HuFlat}.

We work with a definable subcategory $\LL$ of a locally finitely presentable category $\K$. Recall that $\LL$ is $\kappa$-accessible with products and directed colimits for some regular cardinal $\kappa$. \cite[Proposition~5.5]{HuDual} gives that the category of product and $\kappa$-directed colimit preserving functors on $\LL$ is small. Thus $\LL^+$ is a small exact category being a subcategory of this small category.

We begin with a crucial result which is an extension of \cite[Proposition~5.5]{HuDual}. The difference in our version is that the representing object may be external to the definable category.
\begin{pro}\label{key}
Suppose $X\in\LL^+$. There is $K\in\fp\K$ and a regular epimorphism $\eta:\K(K,-)|_\LL\to X$.
\end{pro}

\begin{proof}
(It is interesting to note that, in this proof, we do not use the assumption that $X$ preserves directed colimits.) For each object $K\in\K$, the restriction $y'K:\LL\to\Set$ of the representable functor $\K(K,-)$ to $\LL$ preserves all limits present in $\LL$ and hence, in particular, products. Moreover, if $K$ is finitely presentable, then $y'K$ preserves both products and directed colimits, and hence belongs to $\LL^+=(\LL,\Set)^{\prod\to}$.

Without loss we can assume that both $\K$ and $\LL$ are skeletal so that the full subcategory  $\B:=\mathrm{Pres}_\kappa\LL$ of $\kappa$-presentable objects in $\LL$ is a small category, where $\kappa$ is a regular cardinal such that $\LL$ is $\kappa$-accessible.

Given $X:\LL\to\Set$, consider the small product $\overline{L}:=\prod_{B\in\B}B^{X(B)}$ in $\LL$. Let $J:=\bigsqcup_{B\in\B}X(B)$. Denote for $j\in J$ the product projections by $\pi_j:\overline{L}\to B$. Since $X$ preserves products, we have $X(\overline{L})=\prod_{B\in\B}X(B)^{X(B)}$ and that $X(\pi_j)$ is the product projection in $\LL$. There is $a\in X(\overline{L})$ such that $X(\pi_j)(a)=j$ for all $j\in J$.

Since $\K$ is locally finitely presentable, we can write $\overline{L}$ as the directed colimit, $\colim K_s$, of finitely presentable objects in $\K$ with colimit maps $e_s:K_s\to\overline{L}$. 

There are natural bijections
\begin{eqnarray*}
X(\overline{L}) &\simeq & (\LL,\Set)(y'\overline{L},X)\\
                &\simeq & (\LL,\Set)(y'(\colim K_s),X)\\
                &\simeq & (\LL,\Set)(\lim(y'K_s),X)\\
                &\simeq & \colim(\LL,\Set)(y'K_s,X).
\end{eqnarray*}
The first bijection follows from Yoneda lemma, the second from the fact that the restricted Yoneda embedding $y':\K^{op}\to(\LL,\Set)$ preserves limits, and the final from the universal property of limits.

Thus there is some $s$ and a natural transformation $\eta:y'K_s\to X$ in $\LL^+$ such that $\eta_{\overline{L}}(e_s)=a$. We denote $K_s$ by $K$ and $e_s$ by $e$. Since $X(\pi_j)\circ\eta_{\overline{L}}=\eta_B\circ y'K(\pi_j)$, we have that $j=X(\pi_j)(a)=X(\pi_j)(\eta_{\overline{L}}(e))=\eta_B(\pi_j\circ e)$. Hence for each $B\in\B$ the function $\eta_B$ is surjective. Since every object of $\LL$ is a $\kappa$-directed colimit of objects of $\B$, we observe that $\eta_L$ is surjective for all $L\in\LL$.

Finally observe that $\LL^+$ is Barr-exact and the inclusion $\LL^+\to(\LL,\Set)$ is regular. In the latter $\eta$ being a regular epimorphism means that $\eta_L$ is surjective for all $L\in\LL$. Thus we conclude that $\eta$ is a regular epimorphism.
\end{proof}

We note an immediate consequence of this result.
\begin{cor}\label{Quo}
The functor $Q:\K^+\to\LL^+$ given by precomposition with the inclusion $\LL\to\K$ is an epimorphism in the category $\EX$ of small exact categories and regular functors.
\end{cor}

\begin{proof}
Proposition \ref{key} states that for each $X\in\LL^+$, there is a $K\in\fp\K$ with a regular epimorphism $\rho:\K(K,-)|_\LL\to X$. Thus $\rho$ is a coequalizer of $\rho_1,\rho_2:Y\to\K(K,-)|_\LL$ for some $Y\in\LL^+$. Applying the proposition again for $Y$, we find $K'\in\fp\K$ and a regular epimorphism $\rho':\K(K',-)|_\LL\to Y$. Clearly $\rho$ is also the coequalizer of $\rho_1\circ\rho,\rho_2\circ\rho:\K(K',-)|_\LL\to\K(K,-)|_\LL$. Hence every object of $\LL^+$ is (the codomain of) the coequalizer of arrows between restricted representables.

Suppose $\B$ is a regular category with regular functors $P_1,P_2:\LL^+\to\B$ such that $P_1\circ Q=P_2\circ Q$. Then for each $K\in\fp\K$, we have $P_1(Q(\fp\K(K,-)))=P_2(Q(\fp\K(K,-)))$, equivalently, $P_1(\K(K,-)|_\LL) =P_2(\K(K,-)|_\LL)$. In other words, the functors $P_1$ and $P_2$ agree on the full subcategory of $\LL^+$ consisting of the restricted representables. Since $P_1$ and $P_2$ are regular functors, they preserve coequalizers of kernel pairs. The discussion in the first paragraph gives that coequalizers of the required kernel pairs coincides with the coequalizers of certain pairs of maps between restricted representables. It follows that $P_1\simeq P_2$.
\end{proof}

Now we are ready to state and prove the main result.
\begin{theorem}\label{DefExDef}
For a definable category $\LL$, there is an equivalence $\LL\simeq\Reg(\LL^+,\Set)$. In particular, a category is definable if and only if it is exactly definable. 
\end{theorem}

\begin{proof}
Let $\LL$ be a definable subcategory of a locally finitely presentable category $\K$. Since $\K^+$ is the exact completion of the lex category $\fp\K^{op}$ (Corollary \ref{ExComplLex}), the universal property gives that $\K\simeq\Lex(\fp\K^{op},\Set)\simeq\Reg(\K^+,\Set)$. In other words, for each $K\in\K$, the evaluation functor $ev_K:\K^+\to\Set$ defined by $ev_K(X):=X(K)$ is regular, and each regular functor on $\K^+$ is of this form.

Let $Q:\K^+\to\LL^+$ denote the functor induced by the inclusion of $\LL$ in $\K$. Precomposition with the epimorphism $Q$ gives an embedding of $\Reg(\LL^+,\Set)$ into $\Reg(\K^+,\Set)$. Hence each regular functor $\LL^+\to\Set$ is of the form $ev_K$ for some $K\in\K$.

It is easy to verify that if $K\in\LL$, then $ev_K:\LL^+\to\Set$ is indeed a regular functor. We show the converse, namely that if $ev_K$ factors through $Q$, then $K\in\LL$. Let $G:\LL^+\to\Set$ be the factorization of $ev_K$ through $Q$, i.e., $ev_K=G\circ Q$.

Let $h:K_1\to K_2$ be a morphism in $\M$, where $\LL=\M\mbox{-}\mathrm{Inj}$. Consider the morphism $-\circ h:\K(K_2,-)\to\K(K_1,-)$. It is easy to see that $Q(-\circ h)$ is a regular epimorphism in $\LL^+$. Since $G$ is a regular functor, it maps regular epimorphisms to surjections in $\Set$. Hence $ev_K(-\circ h)=G(Q(-\circ h))$ is a surjection for each $h\in\M$. In other words, $K\in\M\mbox{-}\mathrm{Inj}=\LL$.

The second part follows from the first part and Theorem \ref{ExdefIsDef}. 
\end{proof}

This generalizes the finitary version of \cite[Theorem~5.10]{HuDual}.

We showed in Theorem \ref{ExdefIsDef} that for a small exact category $\B$, there is an equivalence $\B\simeq\Reg(\B,\Set)^+$, whereas in Theorem \ref{DefExDef} we showed that for any definable category $\LL$, there is an equivalence $\LL\simeq\Reg(\LL^+,\Set)$. Given a regular functor $b:\B_1\to\B_2$ between small exact categories, there is an obvious interpretation functor $-\circ b:\Reg(\B_2,\Set)\to\Reg(\B_1,\Set)$. Conversely since products and directed colimits commute with all finite limits and regular epimorphisms in $\Set$ we obtain that, given an interpretation functor $F:\LL_1\to\LL_2$ between definable categories, the  functor $-\circ F:\LL_2^+\to\LL_1^+$ is a regular functor. We summarize our results in the form of a duality below.
\begin{theorem}\label{EXDEFDual}
There is a $2$-equivalence between $2$-categories
\begin{equation*}
\Reg(-,\Set):\EX^{op}\leftrightarrows\DEF:(-,\Set)^{\prod\to}.
\end{equation*}
\end{theorem}

\section{The additive case}
\cite[Proposition~11.1]{PreAMS} states that additive definable categories ``are the same as'' the categories of exact functors on small abelian categories, which in turn gives an appropriate $2$-duality. This is the additive analogue of Theorem \ref{DefExDef}. The goal of this section is to give a new proof of this statement as Theorem \ref{AddDefExDef} using methods developed here to prove its non-additive counterpart. See \cite{Pre32} and \cite{PreAMS} for a detailed exposition of the literature in the additive case.

All the categories in this section are preadditive, and for two preadditive categories $\A,\A'$, the notation $[\A,\A']$ will denote the category of additive functors $\A\to\A'$ and natural transformations between them. Recall that the term `left $\A$-module' describes a covariant functor $\A\to\Ab$ for any small preadditive category $\A$ (i.e., a ring many objects). The notations $\Modl\A$ and $\fpmodl\A$ denote the category of left $\A$-modules and its full subcategory of finitely presented modules respectively.

\subsection{Free abelian completion}
Recall that an additive exact category is necessarily abelian. A functor between two abelian categories is exact if it preserves short exact sequences, or, equivalently, finite limits and colimits. We denote by $\ABEX$ the $2$-category of small abelian categories whose $1$-morphisms are exact functors and $2$-morphisms are natural transformations.

\begin{definition}
The \emph{free abelian completion} of a small preadditive category $\A$ is a small abelian category $Ab(\A)$ together with an additive functor $h:\A\to Ab(\A)$ such that $h$ has the following universal property: for any abelian category $\B$, there is an equivalence of categories $-\circ h:\Ex[Ab(\A),\B]\to[\A,\B]$.
\end{definition}

Since $\Ab$ is a Grothendieck category, by Grothendieck's $AB4^*$ and $AB5$ axioms, product and directed colimit of short exact sequences is again short exact; this is the additive version of the `regular logical doctrine'. Therefore if $\D$ is any additive category with products and directed colimits, then the category $\D^+:=[\D,\Ab]^{\prod\to}$ of functors preserving products and directed colimits is an abelian category.

The existence of the free abelian completion follows from a very general result of Freyd \cite{Freyd}. We present a special case in the form useful for our purposes; the proof relies on a special case of a result we prove in the next section.
\begin{pro}\label{AddExComplLex}
The free abelian completion of a small preadditive category $\A$ is equivalent to $[\Modl\A,\Ab]^{\prod\to}$, the category of functors from $\Modl\A$ to $\Ab$ that preserve products and directed colimits.
\end{pro}

\begin{proof}
\cite[Theorem~4.3]{PreAMS} describes $\fp{[\fpmodl\A,\Ab]}$ as the free abelian category on the small preadditive category $\A$. Thus it is enough to show that an additive functor $F:\Modl\A\to\Ab$ preserves products and directed colimits iff its restriction to $\fpmodl\A$ is finitely presentable.

If $F|_{\fpmodl\A}:\fpmodl\A\to\Ab$ is finitely presentable, then there is an exact sequence $\fpmodl\A(A_2,-)\to\fpmodl\A(A_1,-)\to F|_{\fpmodl\A}\to 0$ for some $A_1,A_2\in\fpmodl\A$. The corresponding sequence $\Modl\A(A_2,-)\to\Modl\A(A_1,-)\to F\to 0$ is also exact. Representables at finitely presentable objects preserve all small limits and hence, in particular, products. Since cokernels commute with products in $\Ab$, we conclude that $F$ preserves products.

Now suppose $F$ preserves products and directed colimits. Proposition \ref{Addkey} applied to $F$ with $\D=\Modl\A$ gives $A_1\in\fpmodl\A$ with a short exact sequence $\Modl\A(A_1,-)\to F\to 0$. The kernel $F_1$ of this epimorphism is again a functor preserving products, and hence applying Proposition \ref{Addkey} to $F_1$ gives a short exact sequence $\Modl\A(A_2,-)\to F_1\to 0$. Thus we have a short exact sequence $\Modl\A(A_2,-)\to\Modl\A(A_1,-)\to F\to 0$. Restricting it to $\fpmodl\A$ gives a finite presentation of $F|_{\fpmodl\A}$.
\end{proof}

\subsection{Definable additive categories}
Let $L_\omega$ denote the finitary logic. Let $R$ denote a ring with unity and $L_R$ denote the one-sorted language of left $R$-modules.

Recall that a positive primitive formula (pp-formula for short) is a formula of the form $\exists\overline{y}\ \theta(\overline{x}\ \overline{y})$ where $\theta$ is a conjunction of atomic formulas. The main result in the model theory of modules is the Baur-Monk pp-elimination theorem which states that every formula in the language $L_R$ is a finite boolean combination of pp-formulas.

\begin{definitions}\cite[\S 3.4.1]{PrePSL}\label{DefAddDef}
A full subcategory $\D$ of $\Modl R$ is \emph{definable} if there is a set $\{\phi_\lambda(x)/\psi_\lambda(x)\}_{\lambda\in\Lambda}$ of pairs of pp-formulas in one variable satisfying $\phi_\lambda(x)\rightarrow\psi_\lambda(x)$  for each $\lambda$ such that $\D$ is precisely the subcategory of modules such that $\phi_\lambda(M)=\psi_\lambda(M)$ for each $\lambda\in\Lambda$.
\end{definitions}

In other words, if $\mathbb{T}$ denotes the above set of pp-pairs, then $\D$ is equivalent to the category $\mathbb{T}\mbox{-}\mathrm{Mod}(\Ab)$ of $\Ab$-models of $\mathbb{T}$, and hence the term.

\begin{theorem}\cite[Theorem~3.4.7]{PrePSL}\label{DefAddChar}
The following conditions on an isomorphism-closed subcategory $\mathcal D$ of $\Modl R$ are equivalent:
\begin{enumerate}
\item $\D$ is definable.
\item $\D$ is closed under products, directed colimits and pure submodules.
\item $\D$ is closed under products, reduced products and pure submodules.
\item $\D$ is closed under products, ultrapowers and pure submodules.
\end{enumerate}
Such a subcategory is, in particular, closed under direct sums and direct summands.
\end{theorem}

(We do not know whether it is possible to obtain (4) of the above theorem as a characterization of non-additive definable categories.)

As a consequence each definable category has an internal theory of purity \cite[Remark~10.3]{PreAMS} which coincides with the external theory (as a subcategory).

More generally, given a small preadditive category $\A$, we say that a full subcategory $\D$ of $\Modl\A$ is definable if it satisfies either of the equivalent conditions (2)-(4) of the theorem above. Furthermore, a definable additive category is a category equivalent to a definable subcategory of $\Modl\A$ for some $\A$.

An additive functor $F:\A\to\A'$ between definable categories is an \emph{interpretation functor} if it preserves products and directed colimits; in notation $F\in[\A,\A']^{\prod\to}$. We denote by $\DEFA$ the $2$-category of definable additive categories whose $1$-morphisms are interpretation functors, and $2$-morphisms are natural transformations.

Since a locally $\lambda$-presentable additive category is locally $\lambda$-presentable, the proof of \cite[Theorem~2.2]{RAB} can be easily adapted to the additive case to obtain the following.
\begin{pro}\label{AddDefInj}
$\D$ is a definable additive category iff it is equivalent to a finite-injectivity class in a module category $\Modl\A$ for a small preadditive category $\A$.
\end{pro}

\subsection{Duality between $\DEFA$ and $\ABEX$}
\begin{definition}\cite{Kra}\label{AddExdef}
A category $\D$ is an \emph{exactly definable additive category} if it is equivalent to the category $\Ex[\B,\Ab]$ of exact functors on a small abelian category $\B$.
\end{definition}

The class of exactly definable additive categories is contained in the class of definable additive categories.
\begin{pro}\cite[Theorem~6.1]{PreAMS}\label{AddExdefIsDef}
For a small abelian category $\B$, the category $\Ex[\B,\Ab]$ of exact functors on $\B$ is a definable subcategory of $\Modl\B$.
\end{pro}

\begin{proof}
Let $\B$ be a small abelian category. The category $\Lex[\B,\Ab]$ of additive lex functors on $\B$ is a finite orthogonality class in $\Modl\B$ \cite[Example~1.33(8)]{AR}. The module category $\Modl\B$ is cocomplete and hence contains all pushouts. Thus by \cite[Remark~4.4(1)]{AR}, $\Lex[\B,\Ab]$ is a finite-injectivity class in $\Modl\B$. The explicit description of the category $\Ex[\B,\Ab]$ as a finite-injectivity class in $\Lex[\B,\Ab]$ is analogous to the proof of Theorem \ref{ExdefIsDef}. Thus $\Ex[\B,\Ab]$ is a finite-injectivity class in $\Modl\B$.
\end{proof}

The correct non-additive analogue of a module category is a presheaf category, so one ought to study definable subcategories of presheaf categories rather than those of locally finitely presentable categories. But a proof similar to the one above shows that we obtain the exactly same class of definable categories in both cases because all locally finitely presentable categories are reflective subcategories of presheaf categories.

We work with a definable subcategory $\D$ of the module category $\Modl\A$. Analogous to the non-additive case, we argue that the category $\D^+$ of product and directed colimit preserving functors on $\D$ is a small abelian category.

The following is the additive analogue of Proposition \ref{key} with essentially the same proof.
\begin{pro}\label{Addkey}
Suppose $X\in\D^+$. There is $K\in\fpmodl\A$ and a short exact sequence $\Modl\A(K,-)|_\D\xrightarrow{\eta}X\to 0$.
\end{pro}

Corollary \ref{Quo} has the following additive analogue with similar proof. A minor simplification in the additive case is the use of kernels instead of kernel pairs.
\begin{cor}\label{AddQuo}
The functor $Q:\Modl\A^+\to\D^+$ given by precomposition with the inclusion $\D\to\Modl\A$ is an epimorphism in the category $\ABEX$ of small abelian categories and exact functors.
\end{cor} 

Now we are ready to state the promised analogue of Theorem \ref{DefExDef}.
\begin{theorem}\cite[Corollary~10.11,\,Proposition~11.1]{PreAMS} \label{AddDefExDef}
For a definable additive category $\D$, there is an equivalence $\D\simeq\Ex[\D^+,\Ab]$. In particular, a category is definable if and only if it is exactly definable. 
\end{theorem}

The proof is analogous to that of Theorem \ref{DefExDef} after one obtains equivalences $\Modl\A\simeq[\A,\Ab]\simeq\Ex[\Modl\A^+,\Ab]$ from the universal property of $\Modl\A^+$ as the free abelian completion of the preadditive category $\A$ (Proposition \ref{AddExComplLex}).

Now we prove the other part of the duality.
\begin{theorem}\cite[Proposition~11.2]{PreAMS}\label{AddMakDua}
For a small abelian category $\B$, we have $\B\simeq[\Ex[\B,\Ab],\Ab]^{\prod\to}$.
\end{theorem}

\begin{proof}
Each small abelian category is a quotient of a free abelian category \cite[2.5]{Ade}. Specifically, given a small abelian category $\B$, there is a quotient functor $Q':Ab(\B)\to\B$ where $Ab(\B)$ is the free abelian category on $\B$ considered as a preadditive category. From the definition of free abelian completion, we get $\Ex[Ab(\B),\Ab]\simeq[\B,\Ab]\simeq\Modl\B$. Moreover, Proposition \ref{AddExComplLex} states that $Ab(\B)\simeq\Modl\B^+$.

The category $\Ex[\B,\Ab]$ is a definable subcategory of $\Modl\B$ as shown in Proposition \ref{AddExdefIsDef}. Thus there is a quotient functor $Q:Ab(\B)(\simeq\Modl\B^+)\to\Ex[\B,\Ab]^+$ from Corollary \ref{AddQuo}. Consider the evaluation functor $ev:\B\to\Ex[\B,\Ab]^+$ defined by $ev_A:F\mapsto F(A)$. It is easy to see that $ev\circ Q'\simeq Q$. Thus $ev$ is an epimorphism.

Theorem \ref{AddDefExDef} states that the exactly definable subcategories of $\Modl\B$ defined by the quotients $\B$ and $\Ex[\B,\Ab]^+$ of $Ab(\B)$ are equivalent; such an equivalence is induced via precomposition with $ev$.

Recall that there is a bijection between abelian quotients of $Ab(\B)$ and definable subcategories of $\Modl\B$ from the equivalence (iii)$\leftrightarrow$(vi) of \cite[Theorem~8.1]{PreAMS}. Hence $ev$ is an isomorphism.
\end{proof}

We showed in Theorem \ref{AddMakDua} that for a small abelian category $\B$, there is an equivalence $\B\simeq\Ex[\B,\Ab]^+$, whereas in Theorem \ref{AddDefExDef} we showed that for any definable additive category $\D$, there is an equivalence $\D\simeq\Ex[\D^+,\Ab]$. Given an exact functor $b:\B_1\to\B_2$ between small abelian categories, there is an obvious interpretation functor $-\circ b:\Ex[\B_2,\Ab]\to\Ex[\B_1,\Ab]$. Conversely since products and directed colimits of short exact sequences are short exact in $\Ab$, the  functor $-\circ F:\D_2^+\to\D_1^+$ induced by an interpretation functor $F:\D_1\to\D_2$ between definable categories is an exact functor. We summarize our results in the form of a duality below.
\begin{theorem}\cite{PreRaj}\label{ABEXDEFADual}
There is a $2$-equivalence between $2$-categories
\begin{equation*}
\Ex[-,\Ab]:\ABEX^{op}\leftrightarrows\DEFA:[-,\Ab]^{\prod\to}.
\end{equation*}
\end{theorem}

\subsection{(Anti)-equivalences between three $2$-categories}
\begin{definitions}
An object $A$ of an abelian category $\A$ with directed colimits is \emph{coherent} if it is finitely presentable and each of its finitely generated subobjects is finitely presentable.

An abelian category $\G$ with directed colimits is \emph{locally coherent} if it has a generating set consisting of coherent objects.

A coherent morphism $f:\G\to\G'$ between locally coherent categories is a pair of adjoint functors $f_*:\G\leftrightarrows\G':f^*$ whose left adjoint maps finitely presentable objects of $\G'$ to finitely presentable objects of $\G$.
\end{definitions}

Let $\COH$ denote the $2$-category of locally coherent categories whose $1$-morphisms are coherent morphisms and $2$-morphisms are natural transformations between them.

The full subcategory $\fp\G$ of a locally finitely presentable category $\G$ is a (skeletally) small abelian category. In the other direction, the free completion of a small abelian category $\A$ under directed colimits, denoted $\Ind\A$ is a locally coherent category.

\begin{figure}[h]
\centering
\begin{tikzpicture}
  \matrix (m) [matrix of math nodes,row sep=5em,column sep=4em,minimum width=8em]
  {\ABEX^{op} & \ & \DEFA \\ \ & \COH & \ \\};
  \path[-stealth]
    (m-1-1) edge[transform canvas={yshift=0.4ex}] node [above] {$\ABEX(-,\Ab)=\Ex[-,\Ab]$} (m-1-3)
    (m-1-3) edge[transform canvas={yshift=-0.4ex}] node [below] {$\DEFA(-,\Ab)=[-,\Ab]^{\prod\to}$} (m-1-1)
    	(m-1-3)	edge[transform canvas={yshift=0.4ex}] node [sloped, above]{$\Fun(-)$} (m-2-2)
    	(m-2-2) edge[transform canvas={yshift=-0.4ex}] node [sloped,below] {$\Abs(-)$} (m-1-3)
    	(m-1-1)	edge[transform canvas={yshift=0.4ex}] node [sloped, above] {$\Ind(-)$} (m-2-2)
    	(m-2-2) edge[transform canvas={yshift=-0.4ex}] node [sloped,below] {$\fp{(-)}$} (m-1-1);
\end{tikzpicture}
\caption{The additive case}\label{AddPic}
\end{figure}
One can associate to a definable subcategory $\D$ of a module category $\Modl\A$ a locally coherent category $\Fun(\D)$ which is a localisation of $[\fpmodl\A,\Ab]$. In the other direction, the full subcategory $\Abs(\G)$ of a locally coherent category $\G$ consisting of the absolutely pure objects (i.e., objects $M$ such that every monomorphism $M\to N$ in $\G$ is pure) is a definable category. See \cite{PreAMS} for the details.

The following theorem summarizes the functorial version of our discussion.
\begin{theorem}\cite[Theorem~1.1]{Pre32}
The triangle in Figure \eqref{AddPic} of $2$-equivalences between $2$-categories commutes.
\end{theorem}

Since the dual of a small abelian category is again so, there is an internal duality at each vertex of the triangle above. For instance, given a ring $R$, there is a bijection between definable subcategories of the categories of left and right $R$-modules. This internal duality is heavily exploited in the proofs of $2$-equivalences in the above diagram, and it was believed that such an internal duality in the non-additive case is necessary to obtain the non-additive analogue of the diagram. But despite the lack of internal dualities, we show its non-additive version in Figure \eqref{NonAddPic}.

\section{Functors on definable categories}
After discussing the duality between regular toposes and small exact categories, we associate one such regular topos, $\Fun_\K(\LL)$, to each definable subcategory $\LL$ of a locally finitely presentable category $\K$. We prove that $\Fun_\K(\LL)$ is independent of the category $\K$ and complete the picture of $2$-(anti-)equivalences between three $2$-categories.
\subsection{Regular toposes}
We define regular toposes and discuss their relationship with small exact categories. This is an unfortunate choice of terminology  since a topos is always a regular category. But a Grothendieck topos is regular iff it occurs as the classifying topos of a regular theory, and hence the name. 
\begin{definitions}\cite[Remark D3.3.10]{Ele}
An object $A$ of a topos $\E$ is called supercompact if every jointly epimorphic family $(f_i:B_i\to A)_{i\in I}$ contains at least one epimorphism $f_i:B_i\to A$.

An object $C$ of a topos $\E$ is called regular if it is supercompact and for any pullback $P\simeq B\times_C B'$ with $B$ and $B'$ supercompact, the object $P$ is supercompact as well.
\end{definitions}

\begin{definitions}\cite[Def.~C2.1.1,\,Theorem~D3.3.1]{Ele}
A coverage $\J$ on a small category $\C$ is regular if it is generated by singleton covering families.

A Grothendieck topos $\E$ is a regular topos if there exists a small lex category $\C$ and a regular coverage $\J$ on $\C$ such that $\E\simeq\Sh\C\J$.

A geometric morphism $e_*:\E_1\leftrightarrows\E_2:e^*$ between two regular toposes is regular if the inverse image functor $e^*$ preserves regular objects.
\end{definitions}
We denote by $\REG$ the $2$-category of regular toposes whose $1$-morphisms are regular geometric morphisms and $2$-morphisms are natural transformations.

\begin{theorem}\cite[Theorem~3.4(b)]{HuFlat}\label{ExComplPSh}
The free exact completion of a small lex category $\C$ can be realised as the full subcategory of the presheaf category $\PSh\C$ of functors $F$ such that for some $C\in\C$ there is a regular epimorphism $\eta:\C(-,C)\to F$ with the following property: the domain $G$ of the kernel pair of $\eta$ admits a regular epimorphism $\C(-,C')\to G$ for some $C'\in\C$.
\end{theorem}

It can be easily verified that the regular objects in $\PSh\C$ are precisely those presheaves on $\C$ which admit a regular epimorphism from a representable with the property in the above theorem (see \cite[D3.3.4]{Ele} for a proof). Hence the exact completion of $\C$ can be realised as the full subcategory $\reg{\PSh\C}$ of regular objects in $\PSh\C$.

We combine the results in Theorems \ref{ExComplPSh} and \ref{ExComplLex} in the following corollary.
\begin{cor}\label{ExComplBoth}
For any small lex category $\C$, there is an equivalence of categories $(\Lex(\C,\Set),\Set)^{\prod\to}\simeq\reg{\PSh\C}$, both equivalent to the free exact completion of $\C$.
\end{cor}

Since regular epimorphisms in a regular category $\B$ are pullback stable, the canonical regular coverage $\reg\T_\B$ generated by singleton families consisting of regular epimorphisms is a regular coverage. Hence $\Sh\B{\reg\T_\B}$ is a regular topos. The canonical regular coverage on an exact category is subcanonincal, i.e., the representables are $\reg\T_B$-sheaves. On the other hand, the full subcategory of regular objects in a regular topos is a small exact category. In fact, the representables $\B(-,B)$  are precisely the regular objects. Below we summarize this discussion.
\begin{theorem}\cite[D3.3.10]{Ele}\label{EXREGDual}
There is a $2$-equivalence between $2$-categories
\begin{equation*}
\Sh{-}{\reg\T}:\EX^{op}\leftrightarrows:\REG:\reg{(-)}.
\end{equation*}
\end{theorem}

The action of these $2$-functors on $1$-morphisms is as follows. Precomposition with a regular functor $b:\B_1\to\B_2$ between exact categories is automatically a morphism of sites $b:(\B_1,\reg\T_{\B_1})\to(\B_2,\reg\T_{\B_2})$ \cite[Ex.~C2.3.2(b)]{Ele}. This induces a geometric morphism $\Sh{\B_2}{\reg\T_{\B_2}}\to\Sh{\B_1}{\reg\T_{\B_1}}$ between corresponding regular toposes whose inverse image functor takes representables to representables, which in turn implies that it is a regular geometric morphism. In the other direction, the restriction of the inverse image functor of a regular geometric morphism between regular toposes is clearly a regular functor.

\begin{rmk}\cite[Remark~D3.3.12]{Ele}
The regular topos $\Sh\B{\reg\T}$ is closed under directed colimits in $\PSh\B$. Thus its regular objects, i.e., the representables, are finitely presentable as objects of $\Sh\B{\reg\T}$. The converse is not true in general since the full subcategory of finitely presentable objects is closed under all coequalizers while the full subcategory of regular objects is only closed under some coequalizers.
\end{rmk}

\subsection{Functor category is independent of the embedding}
We work with a definable subcategory $\LL$ of a locally finitely presentable category $\K$. Denote by $\Fun(\K)$ the collection of all functors $\K\to\Set$ which preserve directed colimits, and by $\fun(\K)$ the full subcategory $\reg{\Fun(\K)}$ of regular objects of $\Fun(\K)$. Clearly $\Fun(\K)\simeq(\fp\K,\Set)$. From Corollary \ref{ExComplBoth}, $\fun(\K)$ is precisely the exact completion of $\fp\K^{op}$ and thus equivalent to the category of all functors $\K\to\Set$ which preserve directed colimits and products. 

Let $\M$ denote the set of morphisms $h$ in $\fp\K$ (assuming $\K$ is skeletal) such that each $L\in\LL$ is injective with respect to $h$. Recall that $\M$ is closed under identities, composition, left cancellation and pushouts along morphisms in $\fp\K$. In other words, $\M$ is a pushout-stable subcategory of $\fp\K$.

To each object $K\in\fp\K$ we associate the collection $\J_\M(K)$ of singleton families consisting of morphisms in $\M$ with domain $K$. Owing to the closure properties of $\M$, it is easy to see that the function $K\mapsto \J_M(K)$ is a regular coverage on $\fp\K^{op}$. We define $\Fun_\K(\LL)$ to be the localization $\Sh{\fp\K^{op}}{\J}$ of $\Fun(\K)$. Since $\fp\K^{op}$ is a small lex category and $\J$ is a regular coverage, the topos $\Sh{\fp\K^{op}}\J$ is a regular topos. Therefore the full subcategory $\reg{\Sh{\fp\K^{op}}\J}$ of regular objects in the sheaf topos $\Sh{\fp\K^{op}}\J$ is an exact category. We denote this category by $\fun_\K(\LL)$.

Unraveling the definition one sees that a functor $F:\fp\K\to\Set$ is a $\J$-sheaf if and only if $F(h)$ is a bijection for each morphism $h:A\to B$ in $\M$.

The equivalences $\K\simeq\Lex(\fp\K^{op},\Set)$ and $(\fp\K^{op})_{ex/lex}\simeq\reg{(\fp\K,\Set)}$ (Corollary \ref{ExComplBoth}) yield $\K\simeq \Lex(\fp\K^{op},\Set)\simeq\Reg((\fp\K^{op})_{ex/lex},\Set)\simeq\Reg(\fun(\K),\Set)$.

\begin{theorem}\label{funcind}
If $\LL$ is a definable subcategory of a locally finitely presentable category $\K$, then $\fun_\K(\LL)\simeq(\LL,\Set)^{\prod\to}$. Hence the functor category $\fun_\K(\LL)$ is independent of the embedding of $\LL$ into $\K$.
\end{theorem}

\begin{proof}
Suppose $\reg\T$ denotes canonical regular coverage on $\LL^+$ consisting of only regular epimorphisms. We claim that there is an equivalence of categories between $\Sh{\LL^+}{\reg\T}$ and $\Sh{\fp\K^{op}}\J$.

Proposition \ref{key} states that every object of $\LL^+$ is a regular epimorphic image of a restricted representable $\K(K,-)|_\LL$ for $K\in\fp\K$. This is equivalent to the statement that the full subcategory $\C$ of $\LL^+$ consisting of the restricted representables is $\reg\T$-dense \cite[Definition~C2.2.1]{Ele}. Let $\reg\T_\C$ denote the restriction of the coverage $\reg\T$ to $\C$. Then by Comparison lemma \cite[Theorem~2.2.3]{Ele}, $\Sh{\LL^+}{\reg\T}\simeq\Sh{\C}{\reg\T_\C}$. So it is enough to exhibit an equivalence of categories between $\Sh{\C}{\reg\T_\C}$ and $\Sh{\fp\K^{op}}\J$.

The corestriction of the Yoneda embedding gives an equivalence (in fact, an isomorphism) of categories $y':\fp\K^{op}\to\C$. It is routine to verify that a presheaf on $\fp\K^{op}$ is a $\J$-sheaf iff the corresponding presheaf on $\C$ induced by $y'$ is a $\reg\T_\C$-sheaf. This completes the proof of the claim.

Thus we have equivalences $\Fun_\K(\LL):=\Sh{\fp\K^{op}}\J\simeq\Sh{\LL^+}{\reg\T}$. Restricting to the full subcategories of regular objects, we have $\fun_\K(\LL)\simeq\reg{\Sh{\LL^+}{\reg\T}}\simeq\LL^+$ where the final equivalence is by Theorem \ref{EXREGDual}.
\end{proof}

In view of Theorem \ref{funcind}, we can remove the subscript $\K$ from the functor category notations. It is useful to explicitly state the equivalence between $\LL^+$ and $\reg{\Sh{\fp\K^{op}}\J}$. Each presheaf $F$ on $\fp\K^{op}$ has a unique extension to a covariant functor $\overrightarrow{F}:\K\to\Set$ that preserves directed colimits. For a $\J$-sheaf $F$, the restriction $\overrightarrow{F}|_\LL$ is the corresponding object of $\LL^+$. Conversely given an object $X\in\LL^+$, the representable functor $\LL^+(-,X)$ is a regular $\reg\T$-sheaf on $\LL^+$. Define $G_X:\fp\K\to\Set$ by $G_X(K):=\LL^+(\K(K,-)|_\LL,X)$.

\subsection{Points of regular toposes}
Recall that a point $e$ of a topos $\E$ is a geometric morphism $e_*:\Set\leftrightarrows\E: e^*$. Let $\Pts(\E)$ denote the category of points of $\E$.

Given a small lex category $\C$ with a coverage $\J$, we say that a functor $\C\to\Set$ is $\J$-continuous if it sends each $\J$-covering family to a jointly epimorphic family.

\begin{theorem}\cite[Corollary~VII.6.4]{MM}
Let $(\C,\J)$ be a site with $\C$ a small lex category. Then the category $\Pts(\E)$ is equivalent to the category of $\J$-continuous lex functors $\C\to\Set$.
\end{theorem}

For a small exact category $\B$ with the canonical regular coverage $\reg\T$, the $\reg\T$-continuous lex functors to $\Set$ are precisely the regular functors on $\B$. Hence the above statement says that $\Pts(\Sh\B{\reg\T})$ is equivalent to $\Reg(\B,\Set)$ which is a definable category by Theorem \ref{ExdefIsDef}. 

A regular geometric morphism $e:\Sh\B{\reg\T}\to\Sh{\B'}{\reg\T}$ between regular toposes is the same thing as a regular functor $e^*:\B'\to\B$ in $\EX$ (Theorem \ref{EXREGDual}), which induces an interpretation functor $-\circ e^*:\Reg(\B,\Set)\to\Reg(\B',\Set)$. In other words, there is a covariant $2$-functor $\Pts(-):\REG\to\DEF$.
\begin{theorem}
There is a $2$-equivalence between $2$-categories
\begin{equation*}
\Pts(-):\REG\leftrightarrows:\DEF:\Fun(-).
\end{equation*}
\end{theorem}
The easy and omitted proof relies on Theorem \ref{funcind}, which can be restated as the commutativity of the triangle in Figure \eqref{NonAddPic} of three $2$-equivalences (filling up the details in Figure \eqref{Scheme}). Compare with Figure \eqref{AddPic}.
\begin{figure}[h]
\centering
\begin{tikzpicture}
  \matrix (m) [matrix of math nodes,row sep=5em,column sep=4em,minimum width=8em]
  {\EX^{op} & \ & \DEF \\ \ & \REG & \ \\};
  \path[-stealth]
    (m-1-1) edge[transform canvas={yshift=0.4ex}] node [above] {$\EX(-,\Set)=\Reg(-,\Set)$} (m-1-3)
    (m-1-3) edge[transform canvas={yshift=-0.4ex}] node [below] {$\DEF(-,\Set)=(-,\Set)^{\prod\to}$} (m-1-1)
    	(m-1-3)	edge[transform canvas={yshift=0.4ex}] node [sloped, above]{$\Fun(-)$} (m-2-2)
    	(m-2-2) edge[transform canvas={yshift=-0.4ex}] node [sloped,below] {$\Pts(-)$} (m-1-3)
    	(m-1-1)	edge[transform canvas={yshift=0.4ex}] node [sloped, above] {$\Sh-{\reg\J}$} (m-2-2)
    	(m-2-2) edge[transform canvas={yshift=-0.4ex}] node [sloped,below] {$\reg{(-)}$} (m-1-1);
\end{tikzpicture}
\caption{The non-additive regular case: detailed}\label{NonAddPic}
\end{figure}

Let us look at the special case of this picture to understand how (the categories of) covariant and contravariant Hom functors interact with each other. Let $\C$ be a small lex category. Then its exact completion, $\C_{ex/lex}$, has the contravariant functor category, $(\C^{op},\Set)$, as the classifying topos whereas the category of covariant lex functors, $\Lex(\C,\Set)$, forms the corresponding definable category. Thus the functors $\Fun(-)$ and $\Pts(-)$ map the categories of covariant and contravariant functors to each other.

Now we make the connection of finitary regular logic to the above triangle more precise.
\begin{definition}\cite{Barr}\label{regtheory}
A regular theory is a theory consisting only of the sentences of the form $\forall\overline{x}(\psi(\overline{x})\rightarrow\phi(\overline{x}))$ where $\phi$ and $\psi$ are regular formulas in the infinitary logic $L_\infty$ (i.e., formulas constructed only using truth, conjunctions and existential quantifiers of appropriate arities).
\end{definition}

A full subcategory $\LL$ in a locally presentable category $\K$ is a small injectivity class iff it is axiomatizable by a regular theory in some logic $L_\lambda$ for a regular cardinal $\lambda$ \cite[Exercise~5.e]{AR}. Inspecting the proof of this statement one can get a cardinal specific version for each regular cardinal $\lambda$. In particular, a definable category, which is a finite-injectivity class in a locally finitely presentable category, is axiomatized by a regular theory in the finitary logic $L_\omega$.

\begin{pro}\cite[Lemma~D1.3.8,\,Proposition~D1.3.10(ii)]{Ele}
Any regular formula is equivalent a pp-formula. Thus any regular theory is equivalently given by the sentences of the form $\forall\overline{x}(\theta(\overline{x})\rightarrow(\exists\overline{y})\theta'(\overline{x}\ \overline{y}))$ where $\theta$ and $\theta'$ are conjunctions of atomic formulas, and we may additionally assume that $\models\theta'(\overline{x}\ \overline{y}) \rightarrow\theta(\overline{x})$.
\end{pro}

To compare with Definition \ref{DefAddDef} of a definable additive category, we see that non-additive definable categories also admit a presentation as the category of models of a theory given by $pp$-pairs.

\section{Future directions}
In a future work we plan to address the `coherent' case where small exact categories are replaced by small pretoposes. The strong conceptual completeness theorem for coherent logic (equivalently, for full first-order logic) was proved by Makkai in \cite{MakStone}, and uses the theory of ultracategories. The corresponding analogue of definable categories would be subcategories of locally finitely presentable categories closed under ultraproducts, directed colimits and pure subobjects.

It is natural to ask whether the syntax-semantics duality for regular logic is entirely `logical' in nature in the sense that it does not depend on the base category. More specifically, it would be useful to replace $\Set$ or $\Ab$ by a sufficiently nice monoidal category $\V$, and a functor by a $\V$-enriched functor. Following \cite{BQ}, we may ask for $\V$ to be locally finitely presentable, exact, and symmetric monoidal closed. Such class of monoidal categories clearly includes $\Set$, $\Ab$, presheaf toposes and module categories.

There are many tools exclusive to the additive case. For example, the fact that the dual category of an abelian category is abelian. Such a statement is not true for small exact categories. The duality $Ab(\A^{op})\simeq(Ab(\A))^{op}$ \cite[Theorem~4.5]{PreAMS} due to Gruson-Jensen, and independently to Auslander, for a small preadditive category $\A$ was crucial in the original proofs of the $2$-duality $\ABEX^{op}\simeq\DEFA$. This provides an internal duality at each vertex of the triangle given in Figure \eqref{AddPic}. The theory of tensor products is also well developed on the additive side.

We believe that the duality for small preadditive categories is more fundamental than the duality for abelian categories. The non-additive analogue of the former duality is simply the statement that the dual of a small category is again one such, which is true. In a future work, we also plan to explore duality and tensor products in the regular non-additive case.

\end{document}